\documentclass[12pt]{article}

\usepackage[margin=1in]{geometry}

\usepackage{amssymb}
\usepackage{amsthm}
\usepackage{amsmath}
\usepackage[]{algorithm2e}
\usepackage[noend]{algpseudocode}

\usepackage{color}
\usepackage{tikz}
\usetikzlibrary{arrows}
\usetikzlibrary{plotmarks}
\usetikzlibrary{shapes,snakes}

\tikzstyle{block} = [draw,fill=blue!20,minimum size=0.5em]
 
\tikzstyle{branch}=[fill,shape=circle,minimum size=3pt,inner sep=0pt]

\newtheorem{alg}{Algorithm}

\newtheorem{theorem}{Theorem}
\newtheorem{thm}[theorem]{Theorem}

\newtheorem{lem}[theorem]{Lemma}

\newtheorem{prop}[theorem]{Proposition}
\newtheorem{cor}[theorem]{Corollary}

\newcommand{\set}[1]{\left\{#1\right\}}
\newcommand{\ignore}[1]{}
\def\tn{\textnormal}

\def\D{\mathcal{D}}

\def\H{\mathcal{H}}
\def\L{\mathcal{L}}

\def\P{\mathcal{P}}

\def\tn{\textnormal}

\def\a{\alpha}
\def\b{\beta}
\def\l{\ell}

\def\pp{\mathinner{\ldotp\ldotp}}

\title{Generalized de Bruijn words \\ for Primitive words and Powers}

\author{Yu Hin Au\\
Department of Mathematics\\
Milwaukee School of Engineering \\
{ au@msoe.edu} \\
\\
}

\begin{document}

\maketitle

\begin{abstract}
We show that for every $n \geq 1$ and over any finite alphabet, there is a word whose circular factors of length $n$ have a one-to-one correspondence with the set of primitive words. In particular, we prove that such a word can be obtained by a greedy algorithm, or by concatenating all Lyndon words of length $n$ in increasing lexicographic order. We also look into connections between de Bruijn graphs of primitive words and Lyndon graphs.

Finally, we also show that the shortest word that contains every $p$-power of length $pn$ over a $k$-letter alphabet has length between $pk^n$ and roughly $(p+ \frac{1}{k}) k^n$, for all integers $p \geq 1$. An algorithm that generates a word which achieves the upper bound is provided.
\end{abstract}

\section{Introduction}

In this paper, we study generalizations of de Bruijn words, and provide a few results related to some well-studied collection of words. We first establish some notation. Given an integer $k \geq 2$, we define $\Sigma_k := \set{0, 1, \ldots, k-1}$, and let $|w|$ denote the length of any finite word $w \in \Sigma_k^{*}$. Also, we define $w[i]$ to be the $i^{\tn{\scriptsize{th}}}$ symbol in $w$, and $w[i\pp j]$ to be the word $w[i]w[i+1] \cdots w[j-1]w[j]$, for any indices $i,j$ such that $1 \leq i \leq j \leq |w|$. If $i > j$, then we define $w[i\pp j]$ to be the empty word. Also, given any word $x \in \Sigma_k^n$ and an integer $p \geq 1$, we define $x^p$ to be the word obtained from concatenating $p$ copies of $x$. For example, $(01)^3 = 010101$. A word $w$ is \emph{$p$-power} if $w = x^p$ for some word $x$ and some integer $p$. Conventionally, $2$-powers are usually called squares, and $3$-powers are called cubes. 

We say that a word $x$ is a \emph{factor} (also sometimes called a subword) of another word $w$ if $x= w[i\pp j]$ for some indices $i,j$, and we say that $x$ is a \emph{circular factor} of $w$ if $x$ is a factor of $w^p$ for some integer $p$. Given integers $n$ and $k$, a sequence in which every word in $\Sigma_k^n$ appears as a circular factor exactly once is called a \emph{de Bruijn word}, named after Nicolaas Govert de Bruijn for his work on these sequences in~\cite{deBruijn46a}. For example, $00011101$ is a de Bruijn word for $\set{0,1}^3$. It has long been known that such a sequence exists for $\Sigma_k^n$, for every $n,k\geq 1$. In fact, there are exponentially many such sequences~\cite{FlyeSainteMarie1894a}.

There are many ways to generate a de Bruijn word for $\Sigma_k^n$. First, one can be obtained by a greedy algorithm:

\begin{alg}\label{babygreedy}
Generating a de Bruijn word $w$ for $\Sigma_k^n$

\begin{algorithm}[H]
\DontPrintSemicolon 
\KwIn{Integers $n,k \geq 1$}
Set $w[1 \pp n] = 0^n$\;
Set $i = n +1$\;
\While{$\exists \a \in \Sigma_k$ such that $w[i-n+1 \pp i-1] \a$ is not a factor of $w[1 \pp i-1]$}
{
Set $w[i]$ to be the largest such symbol $\a$\;
Increment $i$\;
}
Discard last $n-1$ symbols in $w$\;
\Return{$w$}
\end{algorithm}
\end{alg}

In other words, we start with $0^n$, and then successively append the largest symbol in the alphabet that does not create a factor of length $n$ that had appeared earlier in our sequence, and stop if there is no such symbol. Then the resulting word, with the last $n-1$ symbols removed, is a de Bruijn word for $\Sigma_k^n$. This simple algorithm was discovered independently by several mathematicians~\cite{Fredricksen82a}, first by~\cite{Martin34a}. 

Alternatively, one can also construct a de Bruijn word for $\Sigma_k^n$ by doing the following. Given a word $w \in \Sigma_k^n$, define 
\[
w^{(i)} := w[i+1\pp n]x[1\pp i]
\]
for all $i = 1, \ldots, n$. We say that $w^{(1)}, \ldots, w^{(n)}$ are the \emph{conjugates} of $w$, and define a word $w \in \Sigma_k^n$ to be \emph{primitive} if $w \neq w^{(i)}$ for all $i \in \set{1,2,\ldots, n-1}$. Next, a word $w \in \Sigma_k^n$ is \emph{Lyndon} if $w$ is primitive, and is the lexicographically smallest among its conjugates. The following result, due to Fredricksen and Maiorana~\cite{FredricksenM78a}, establishes a remarkable connection between de Bruijn words and Lyndon words.

\begin{thm}\label{LyndonO}
Let $w$ be the concatenation of all Lyndon words in $\Sigma_k^*$ of length dividing $n$, in increasing lexicographic order. Then $w$ is a de Bruijn word for $\Sigma_k^n$.
\end{thm}

For instance, the six binary Lyndon words with length dividing four are, in increasing lexicographic order, $0, 0001, 0011, 01,0111$ and $1$. Thus, by Theorem~\ref{LyndonO},
\[
w := 0000100110101111
\]
is a de Bruijn word for $\set{0,1}^4$. An advantage of this approach is that, unlike the greedy algorithm that requires exponential storage space during its execution, generating a de Bruijn word by concatenating Lyndon words can be done in constant time and space per bit~\cite{RuskeySW92a}.

More recently, Moreno~\cite{Moreno05a} extended the notion of de Bruijn words to an arbitrary dictionary $\mathcal{D} \subseteq \Sigma_k^n$, and defined a de Bruijn word for $\mathcal{D}$ to be a sequence in which every word in $\mathcal{D}$ (and no other words in $\Sigma_k^n$) appears as a circular factor exactly once. For instance, if we let $\D$ be the set of words in $\set{0,1}^4$ with at least two 1s, then the word $11101011001$ is a de Bruijn word for $\D$. Yet further generalizations of de Bruijn words, such as~\emph{universal cycles}, have also been studied in the literature (see, for instance,~\cite{ChungDG92a} and~\cite{Johnson09a}).

This paper will be organized as follows: In the next section, we first work with Moreno's generalization, and show that de Bruijn words of the set of primitive words in $\Sigma_k^n$ exist, for all integers $n,k \geq 2$. Among other results, we prove that a de Bruijn word for the set of primitive words in $\Sigma_k^n$ can be generated by either of the following procedures:

\begin{itemize}
\item
Start with $w = 0^{n-1}$, and iteratively append the largest symbol in $\Sigma_k$ that does not create a factor of length $n$ that is not primitive or has already appeared in $w$. Stop when the word cannot be further extended, and discard the last $n-1$ symbols of $w$.
\item
Concatenate all Lyndon words of length $n$, in increasing lexicographic order.
\end{itemize}

Some of the tools we use, such as presenting greedy algorithms under the framework for preference functions and making connections between de Bruijn and Lyndon graphs of dictionaries, could help with the analysis and construction of de Bruijn words of other dictionaries. In Section 3, we look into a different generalization of de Bruijn words, and show that the shortest sequence that contains all $p$-powers of length $pn$ as factors has length between $pk^n$ and roughly $(p+ \frac{1}{k})k^n$, for all integers $p \geq 1$. We provide an algorithmic proof for the upper bound, and discuss some computational results.

\ignore{

Given any word $x \in \Sigma_k^n$ and a number $p \geq 1$ such that $pn$ is an integer, we define
\[
x^p := \underbrace{x \cdot x \cdots x}_\text{$\lfloor p \rfloor$ times} \cdot x[1\pp (p - \lfloor p \rfloor)n].
\]
A word $w \in \Sigma_k^n$ is \emph{primitive} if it cannot be expressed as $x^p$ for another word $x$ and integer $p \geq 2$. Equivalently, $w$ is primitive if there does not exist any $i \in \set{1,2,\ldots, n-1}$ such that $w = w[i+1\pp n]w[1\pp i]$. Also, given any number $p \geq 1$ such that $pn$ is an integer, we say that $w \in \Sigma_k^{pn}$ is a \emph{$p$-power}  if there exists a word $x \in \Sigma_k^n$ such that $w = x^p$. Primitive words and powers are among the most fundamental and well-studied families of words, due to their wealth of theoretical properties and omnipresence in real-world applications (see, for example,~\cite{Lothaire05a}).

Here, we look into the problem of finding the shortest sequence that contains all primitive words of a given length over a fixed alphabet. We will provide multiple ways to generate such a sequence, showing an interplay of many well-known concepts in combinatorics on words, such as de Bruijn words and Lyndon words. Finally, we tackle the same problem with primitive words replaced by $p$-powers, and provide some bounds on the length of such a sequence.
}

\section{de Bruijn Words for Primitive Words}

First of all, it is apparent de Bruijn words do not exist for some dictionaries $\D \subseteq \Sigma_k^n$.  For instance, consider the dictionary $\D := \set{0000, 0001,0011,0111}$. There is clearly no binary word of length $4$ that contains all four words in $\D$ as circular factors. Moreno~\cite{Moreno05a} observed that the dictionaries for which de Bruijn words exist can be characterized by looking at their corresponding de Bruijn graphs. Given $\D \subseteq \Sigma_k^n$, its \emph{de Bruijn graph} $G^{\D}$ is defined as follows:
\begin{itemize}
\item
Its vertices $V(G^{\D})$ is the set of words in $\Sigma_{k}^{n-1}$ that are factors of some word in $\D$;
\item
Its arcs $E(G^{\D})$ is the set of ordered pairs $\set{u,v}$ where $u,v \in \Sigma_k^{n-1}$ and there exists a word in $\D$ whose prefix is $u$ and suffix is $v$.
\end{itemize}
For example, Figure~\ref{fig1} illustrates $G^{\D}$ where $\D$ is the set of words in $\set{0,1}^4$ with at least two 1s. Each arc $\set{u,v}$ (which will sometimes be abbreviated as $uv$ from here on to reduce cluttering) is labelled by the unique word in $\D$ of which $u$ is a prefix and $v$ is a suffix. Alternatively, $G^{\D}$ can be defined as the de Bruijn graph of $\Sigma_k^n$, with arcs corresponding to words in $\Sigma_k^n \setminus \D$ removed, and then isolated vertices deleted.

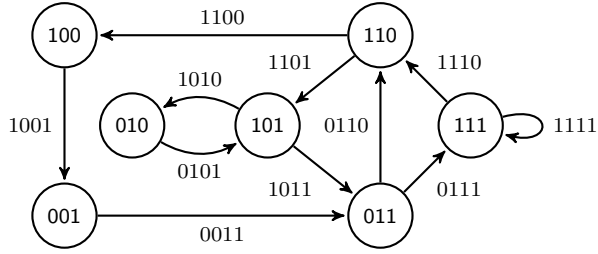
\begin{figure}
\begin{center}
\begin{tikzpicture}[scale =0.6 ,>=stealth',shorten >=1pt,auto,node distance=4cm,
  thick,main node/.style={circle, draw,font=\scriptsize\sffamily}]

  \node[main node] at (0,0) (001) {001};
  \node[main node] at (1.5,2) (010) {010};
  \node[main node] at (7,0) (011) {011};
  \node[main node] at (0,4) (100) {100};
  \node[main node] at (4.5,2) (101) {101};
  \node[main node] at (7,4) (110) {110};
  \node[main node] at (9,2) (111) {111};
;
 \path[->, font=\scriptsize\sffamily] 
(001) edge node [below] {$0011$} (011)
 (010) edge [bend right] node [below] {$0101$} (101)
(011) edge node {$0110$} (110)
 (011) edge node [below right] {$0111$} (111)
 (100) edge node [left] {$1001$} (001)
 (101) edge [bend right] node [above] {$1010$} (010)
 (101) edge node [below left]{$1011$} (011)
 (110) edge node [above]{$1100$} (100)
 (110) edge node [above left]{$1101$} (101)
(111) edge node [above right] {$1110$} (110)
  (111) edge [loop right] node {$1111$} (111);
\end{tikzpicture}
\caption{The de Bruijn graph for the set of words in $\set{0,1}^4$ with at least two 1s.}\label{fig1}
\end{center}
\end{figure}

Given a directed graph $G$, an \emph{Eulerian cycle} in $G$ is a closed walk that uses every arc in $G$ exactly once. An important property of de Bruijn graphs is that, for any dictionary $\D \subseteq \Sigma_k^n$, there is a one-to-one correspondence between de Bruijn words of $\D$ and the Eulerian cycles of $G^{\D}$~\cite{Moreno05a}. For instance, an Eulerian cycle in the graph in Figure~\ref{fig1} can be obtained from starting at the vertex $001$, and going through arcs  $0011$, $0111$, $1111$,$1110$, $1101$, $1010$, $0101$, $1011$, $0110$, $1100$, and $1001$ in that order. Then by concatenating the last symbol in each of these arcs, we obtain $11101011001$, the aforementioned de Bruijn word for this dictionary. Likewise, given any de Bruijn word, one can construct from its circular factors a corresponding Eulerian cycle in the de Bruijn graph.

Next, we show that there is a de Bruijn word for the set of primitive words in $\Sigma_k^n$, for every $n,k \geq 2$. In fact, we will provide three rather different proofs, as they each make use of different tools and connects with different existing results.

\subsection{Using Greedy Algorithms}

Before we focus on the set of primitive words, we look into a general framework that will allow us to analyze the viability of generating de Bruijn words using greedy algorithms for arbitrary dictionaries. Given a dictionary $\D \subseteq \Sigma_k^n$, Moreno~\cite{Moreno05a} showed that a necessary condition for $\D$ to have a de Bruijn word is the following:
\begin{equation}\label{C2}
|\set{ \alpha \in \Sigma_k : \alpha u \in \mathcal{D}}| = |\set{ \alpha 
\in \Sigma_k : u \alpha \in \mathcal{D}}|, \quad  \forall u \in \Sigma_k^{n-1}.
\end{equation}

That is, for any word $u$ of length $n-1$, the number of symbols that can left-extend $u$ to a word in $\D$ is equal to the number of symbols that can right-extend $u$ to a word in $\D$. This is equivalent to the condition that the in-degree is equal to the out-degree for every vertex in the graph $G^{\D}$.

\ignore{
Unless otherwise stated, we will let $\D$ denote the set of primitive words in $\Sigma_k^n$ in this section. Let's first formalize a few notions. }

Next, given a dictionary $\D$, we say that a word $u \in \Sigma_k^{*}$ is \emph{$\D$-nonrepeating} if it satisfies all of the following conditions:
\begin{enumerate}
\item
$|u| \geq n-1$, and $u[1 \pp n-1]$ is a factor of some word in $\D$;
\item
$u$ does not contain any word in $\Sigma_k^n \setminus \D$ as a factor;
\item
$u$ does not contain any word in $\D$ as a factor more than once.
\end{enumerate}

Note that if $x \in \D$, then $x$ and $x[1 \pp n-1]$ are both $\D$-nonrepeating. Also, using the same correspondence between de Bruijn words of $\D$ and Eulerian cycles in $G^{\D}$ described previously, a $\D$-nonrepeating word translates to a walk in $G^{\D}$ in which no arc is used more than once. As we will see subsequently, these $\D$-nonrepeating words will serve as eligible starting points of constructing de Bruijn words for $\D$.

Next, let $\P$ be a \emph{preference function} that maps each word in $\Sigma_k^{n-1}$ to an ordered set that contains each symbol in $\Sigma_k$ exactly once. We then define $f_{\max}(u)$ to be the word generated by the following algorithm

\begin{alg}
Generating $f_{\max}(u)$

\begin{algorithm}[H]
\DontPrintSemicolon 
\KwIn{Dictionary $\D \subseteq \Sigma_k^n$, preference function $\P$,  $\D$-nonrepeating word $u$}
Set $f_{\max}(u)[1 \pp |u|] = u$\;
Set $i = |u| +1$\;
\While{$\exists \a \in \Sigma_k$ such that $f_{\max}(u)[i-n+1 \pp i-1] \a \in \D$ and is not a factor of $f_{\max}(u)[1 \pp i-1]$}
{
Set $f_{\max}(u)[i]$ to be the first such symbol in the set $\P( f_{\max}(u)[i-n+1 \pp i-1])$\;
Increment $i$\;
}
\Return{$f_{\max}(u)$}
\end{algorithm}

\ignore{
\begin{enumerate}
\item
Set $f_{\max}(u)[i] = u[i],~\forall i = 1,\ldots, |u|$;
\item
$i = |u| + 1$;
\item
If there exists a symbol $\a \in \Sigma_k$ where $f_{\max}(u)[i-n+1] \a$ is primitive, set $f_{\max}(u)[i]$ to be the first such symbol in the set $\P( f_{\max}(u)[i-n+1 \pp i-1])$. Increment $i$ and repeat Step 3.
\item
Otherwise, terminate and return $f_{\max}(u)$.
\end{enumerate}}
\end{alg}

For example, let $\D = \set{0,1}^4$, $u = 0000$, and $\P$ be the preference function where
\[
\P(w) = \set{1,0}, \quad \forall w \in \set{0,1}^3.
\]
In other words, when choosing a symbol to append to $f_{\max}(u)$, we always try the symbol $1$ before $0$. In this case, $f_{\max}(u) = 0000111101100101000$, and removing the last $3$ symbols result in  a de Bruijn word for $\D$. More generally, when $\D = \Sigma_{k}^n$, $u = 0^n$ and
\[
\P(w) = \set{k-1, k-2, \ldots, 1, 0}, \quad \forall w \in \Sigma_k^{n-1},
\]
the construction of $f_{\max}(u)$ (with the last $n-1$ symbols removed) coincides with Algorithm~\ref{babygreedy}, the aforementioned greedy algorithm that generates a de Bruijn word for $\Sigma_k^n$. Here, the preference function $\P$ can be interpreted as always attempting to pick the largest eligible symbol to extend $f_{\max}(u)$. While the framework with preference functions may seem a little clumsy at this point, it allows the possibility of having the preference of symbols vary upon the current suffix of $f_{\max}(u)$, which we shall explore later in this section.

We now characterize situations where, given dictionary $\D$, $\D$-nonrepeating word $u$, and preference function $\P$, $f_{\max}(u)$ is in fact a de Bruijn word for $\D$ (after having its last $n-1$ symbols removed). Consider the following closely related sequence:

\begin{alg}
Generating $f_{\min}(u)$

\begin{algorithm}[H]
\DontPrintSemicolon 
\KwIn{Dictionary $\D \subseteq \Sigma_k^n$, preference function $\P$, $\D$-nonrepeating word $u$}
Set $f_{\min}(u)[1 \pp |u|] = u$ \;
Set $i = |u| +1$\;
\While{$\exists \a \in \Sigma_k$ such that $f_{\min}(u)[i-n+1 \pp i-1] \a \in \D$ and is not a factor of $f_{\min}(u)[1 \pp i-1]$}
{
Set $f_{\min}(u)[i]$ to be the last such symbol in the set $\P( f_{\min}(u)[i-n+1 \pp i-1])$\;
Increment $i$\;
}
\Return{$f_{\min}(u)$}
\end{algorithm}
\end{alg}

That is, $f_{\min}(u)$ is constructed in a similar fashion as $f_{\max}(u)$, except that we iteratively append the least preferred symbol among all eligible ones, instead of the most preferred. Somewhat surprisingly, the words obtained from being greedy and ``anti-greedy" can be related as follows.

\begin{thm}\label{greedyok}
Suppose we are given a dictionary $\D \subseteq \Sigma_k^{n}$ that satisfies~\eqref{C2}, $u \in \Sigma_{k}^{*}$ that is $\D$-nonrepeating, and preference function $\P$. If $u[1 \pp n-1]$ is a factor of $f_{\min}(w)$ for all $w \in \Sigma_{k}^{n-1}$ that is a factor of some word in $\D$, then $f_{\max}(u)$ contains every word in $\D$ as factor exactly once. Moreover, the word obtained from $f_{\max}(u)$ by discarding the last $n-1$ symbols is a de Bruijn word for $\D$.
\end{thm}

\begin{proof}
By construction (and the fact that $u$ is $\D$-nonrepeating), every factor of $f_{\max}(u)$ of length $n$ is in $\D$, and no such factors appear twice. Therefore, it suffices to show that every word in $\D$ does appear as a factor in $f_{\max}(u)$.

First, observe that $f_{\max}(u)$ must end with $u[1 \pp n-1]$. Otherwise, let $x$ be the suffix of $f_{\max}(u)$ of length $n-1$, and suppose $x$ appears $q$ times in $f_{\max}(u)$ as a factor. The construction of $f_{\max}(u)$ terminates at $x$ implies that $|\set{\b : x \b \in \mathcal{D}}| = q-1$. However, since $f_{\max}(u)$ starts with $u[1 \pp n-1]$ which by assumption is not equal to $x$, we have $|\set{\b : \b x \in \mathcal{D}}| \geq q$, contradicting the assumption that $\D$ satisfies~\eqref{C2}. 

Next, suppose for a contradiction that there exists $\a_1 \in \Sigma_k, y \in \Sigma_k^{n-1}$ such that $\a_1y \in \D$ but is not a factor of $f_{\max}(u)$. Since $| \set{\b : \b y \in \mathcal{D}}| = | \set{\b : y \b \in \mathcal{D}}|$ and
\[
| \set{\b : \b y~\textrm{is a factor of $f_{\max}(u)$}}| = |\set{\b : y \b~\textrm{is a factor of $f_{\max}(u)$}}|,
\]
there exists $\a_2 \in \Sigma_k$ such that $y\a_2 \in \D$ but is not a factor of $f_{\max}(u)$. In particular, since the algorithm always chooses the most preferred symbol to extend $f_{\max}(u)$, we may assume that $\a_2$ is the last symbol in the ordered set $\P(y)$ where $y\a_2$ is in $\D$.

Applying the same reasoning on $y[2\pp n-1]\a_2$, we conclude that if we let $\a_3$ be the least preferred symbol in $\P(y[2\pp n-1]\a_2)$ such that $y[2\pp n-1]\a_2\a_3$ is in $\D$, then $y[2\pp n-1]\a_2 \a_3$ does not appear in $f_{\max}(u)$. 

Keep proceeding in this manner, and we conclude that any factor of length $n$ in $f_{\min}(\a_1y)$ does not appear in $f_{\max}(u)$. By the same argument we used above to show that $f_{\max}(u)$ must have $u[1\pp n-1]$ as its prefix and suffix, we may conclude that $f_{\min}(\a_1y)$ has both $\a_1 y[1\pp n-1]$ as prefix and suffix. Since $f_{\min}(\a_1 y)$ contains $u[1 \pp n-1]$ as a factor by assumption, this implies that there exists symbol $\b$ where $u[1 \pp n-1] \b$ is both in $\D$ and a factor of $f_{\min}(\a_1y)$, and thus $u[1 \pp n-1]\b$ does not appear in $f_{\max}(u)$. However, since we have shown above that $f_{\max}(u)$ must end with $u[1 \pp n-1]$, it then must contain all words in $\D$ with prefix $u[1 \pp n-1]$, and thus we obtain a contradiction. Therefore, $f_{\max}(u)$ must contain every word in $\D$ as a factor exactly once. Finally, since $f_{\max}(u)$ both starts and ends with $u[1 \pp n-1]$, a de Bruijn word for $\D$ can be obtained by discarding the last $n-1$ symbols of $f_{\max}(u)$.
\end{proof}

We remark that the converse of Theorem~\ref{greedyok} is not true. For an example, let $\D = \set{0,1}^4$ and $\P(w) = \set{1,0}$ for all $w \in \set{0,1}^3$, then
\[
f_{\max}(0011) = 0011110110010100001,
\]
and removing the last $3$ symbols result in a de Bruijn word for $\set{0,1}^4$. However, we see that
\[
f_{\min}(000) = 00001000,
\]
which does not contain $0011$. Hence, while $f_{\min}(w)$ contains $u$ for every $w \in \Sigma_k^{n-1}$ is a sufficient condition for $f_{\max}(u)$ to contain a de Bruijn word for $\D$, it is not necessary. 

Next, we apply Theorem~\ref{greedyok} to show that the simple greedy algorithm that generates a de Bruijn word for $\Sigma_k^n$ can be adapted to generate a de Bruijn word for the set of primitive words. We first need the following.

\begin{lem}\label{C2prim}
Let $\D$ be the set of primitive words in $\Sigma_k^n$. Then $\D$ satisfies~\eqref{C2}.
\end{lem}

\begin{proof}
For any $u \in \Sigma_k^{n-1}, \alpha \in \Sigma_k$,  if $\alpha u$ is not primitive, then it can be written as $(\a x)^p$ for some word $x$ and integer $p \geq 2$. But then $u \a = (x \a)^p$ is not primitive either. Thus, we see that for every $u \in \Sigma_k^{n-1}$, $\a u $ is primitive if and only if $u \a$ is primitive. 

Therefore, the sets on either side of the equality in~\eqref{C2} are identical for every $u \in \Sigma_k^{n-1}$, so it is apparent that they have the same size.
\end{proof}

We will also need the following property of primitive words:

\begin{lem}\label{u0u1}
For every $u \in \Sigma_k^{n-1}$ and distinct symbols $\a,\b \in \Sigma_k$, if $u \a$ is not primitive, then  every factor of $u\b^{n-1}$ of length $n$ is primitive.
\end{lem}

\begin{proof}
To obtain a contradiction, suppose that $u\a$ is not primitive, and that there exists integer $\l \leq n-1$ such that $u[\l \pp n-1]\b^{\l}$ is also not primitive. Then there exist words $x,y$ and integers $p,q \geq 2$ such that $u \a = x^p$ and $\b^{\l-1} u[\l \pp n-1]\b = y^q$ (the latter is due to $\b^{\l-1}u[\l \pp n-1]\b$ being a conjugate of $u[\l \pp n-1]\b^{\l}$). Notice that $|y| > \l$, or otherwise $\b^{\l-1} u[\l \pp n-1]\b = y^q$ implies $y = \b^{|y|}$, and consequently $u = \b^{n-1}$, which would imply that $u\a$ is primitive. Thus, we obtain that 
\begin{eqnarray}
\label{aeq} u[s|x|] &=& \a, \quad \forall s \in \set{1, \ldots, p-1}, \\
\label{beq} u[t|y| +r] &=& \b, \quad \forall t \in \set{1,\ldots, q-1}, r \in \set{0, \ldots, \l-1}.
\end{eqnarray}

Define $m$ to be the least common multiple of $|x|$ and $|y|$. If $m < n$, then $u[m] = \a$ by~\eqref{aeq} and $u[m] = \b$ by~\eqref{beq}, a contradiction. Thus, $|x|$ and $|y|$ are coprime, and so for any fixed $r \in \set{1,\ldots, |y|-1}$, there exists $s \in \set{1, \ldots, |y|-1}$ such that $s|x| \equiv r$~(mod $|y|$). Since $u[s|x|] = \a$ for all $s \in \set{1,\ldots, |y|-1}$, this implies that $y = \a^{|y|-1} \b$. But then $u\a = \left(  \a^{|y|-1} \b \right)^{q-1} \a^{|y|}$ would be primitive, which is a contradiction.
\end{proof}

We are finally ready to prove the following:

\begin{thm}\label{primgreedy}
Let $\D$ be the set of primitive words in $\Sigma_k^n$ where $n,k \geq 2$, and let $\P$ be the preference function where
\[
\P(w) = \set{k-1, k-2, \ldots, 1 ,0}, \quad \forall w \in \Sigma_{k}^n.
\]
Then $f_{\max}(0^{n-1})$ (minus the last $n-1$ symbols) is a de Bruijn word for $\D$.
\end{thm}

\begin{proof}
First, $0^{n-1}$ is obviously $\D$-nonrepeating. Also, we have shown that the set of primitive words satisfies~\eqref{C2}. Thus, by Theorem~\ref{greedyok}, it suffices to show that $f_{\min}(w)$ contains $0^{n-1}$ for all $w \in \Sigma_k^{n-1}$. By Lemma~\ref{u0u1}, we see that $f_{\min}(w)$ either has prefix $w 0^{\l}$ that contains a factor of $0^{n-1}$, or $w0^{\l}1^{n-1}0^{n-1}$ for some $\l \geq 0$. In either case, $f_{\min}(w)$ contains $0^{n-1}$, and our claim follows.
\end{proof}

Thus, we have shown that starting with $0^{n-1}$ and iteratively appending the largest possible symbol that does not create a factor of length $n$ that has already appeared or is not primitive will result in a de Bruijn word for the set of primitive words. it is not hard to see that the ingredients in the above arguments can be extended to show the following slightly stronger result:

\begin{thm}
let $\D$ be the set of primitive words in $\Sigma_k^n$, where $n,k \geq 2$. Let $\P$ be the preference function such that
\[
\P(w) = \set{\a_1, \a_2, \ldots, \a_{k-1}},~\quad \forall w \in \Sigma_k^n,
\]
where $\set{\a_1, \ldots, \a_{k-1}}$ is any fixed ordering of the alphabet $\Sigma_k$. Then $f_{\max}\left( (\a_{k-1})^{n-1} \right)$ (minus the last $n-1$ symbols) is a de Bruijn word for $\D$.
\end{thm}

In particular, this implies that the ``prefer minimum'' algorithm (start with $n-1$ copies of the largest symbol, iteratively extend sequence by writing down the smallest symbol that does not create a repeat or non-primitive factor of length $n$) also generates a de Bruijn word. 

We next look into a case where the preference function $\P$ varies upon $w \in \Sigma_k^{n-1}$. First, Alhakim~\cite{Alhakim10a} showed the following interesting result for binary sequences, which we paraphrase here using preference functions:

\begin{thm}
Let $\D = \set{0,1}^n$, and $\P$ be the preference function such that
\[
\P(w) = \left\{
\begin{array}{ll}
\set{1,0} & \tn{if $w \in \set{0,1}^{n -1}$ ends with a $0$;}\\
\set{0,1} & \tn{if $w \in \set{0,1}^{n -1}$ ends with a $1$.}\\
\end{array}
\right.
\]
Then $f_{\max}(0^{n})$, with the last $n-1$ symbols removed and then the symbol $1$ appended, is a de Bruijn word for $\set{0,1}^n$.
\end{thm}

Alhakim named the construction of this sequence the ``prefer opposite algorithm'' --- at each iteration, it prefers to extend the sequence by adding the symbol that is different from the current last symbol in the sequence. For an example, when $n =4$, we obtain
\[
f_{\max}(0000) = 000010100110111000.
\]
Then we remove the last three $0$'s and add a $1$, and obtain $0000101001101111$, which is a de Bruijn word for $\set{0,1}^4$.

We now apply Theorem~\ref{greedyok} again to show that a de Bruijn word for the set of primitive words can be obtained in this ``prefer opposite'' manner as well. 

\begin{thm}\label{preferopposite}
Let $\D$ be the set of primitive words in $\set{0,1}^n$, and define the preference function $\P$ such that
\[
\P(w) = \left\{
\begin{array}{ll}
\set{1,0} & \tn{if $w \in \set{0,1}^{n -1}$ ends with a $0$;}\\
\set{0,1} & \tn{if $w \in \set{0,1}^{n -1}$ ends with a $1$.}\\
\end{array}
\right.
\]
Then $f_{\max}(0^{n-1})$, with the last $n-1$ symbols removed, is a de Bruijn word for $\D$.
\end{thm}

\begin{proof}
Again, $0^{n-1}$ is $\D$-nonrepeating, and the set of primitive words satisfies~\eqref{C2}. Next, consider $f_{\min}(w)$, which intuitively is the word obtained from iteratively extending $w$ with primitive factors in a ``prefer same'' manner. It only remains to show that $f_{\min}(w)$ contains $0^{n-1}$ for all $w \in \set{0,1}^{n-1}$. Let $\l \geq n$ be the smallest integer such that $f_{\min}(w)[\l] \neq f_{\min}(w)[\l+1]$. Such an $\l$ must exist, as the algorithm would not produce a non-primitive factor of length $n$, and thus would not append the same symbol $n$ consecutive times.

Next, $f_{\min}(w)[\l] \neq f_{\min}(w)[\l+1]$ means that setting $f_{\min}(w)[\l+1] = f_{\min}(w)[\l] $ would have created a non-primitive factor (as the construction of $f_{\min}(w)$ ``prefers same''). Thus, by Lemma~\ref{u0u1}, $f_{\min}(w)[\l+1 \pp \l+n] = (f_{\min}(w)[\l+1])^{n-1}$. Now if $f_{\min}(w)[\l] =1$, then we have our factor of $0^{n-1}$ in $f_{\min}(w)$. Otherwise, if $0^{n-1}$ had not shown up earlier in $f_{\min}(w)$ already, $f_{\min}(w)[\l+n]$ would be followed by a string of $n-1$ $0$'s (by Lemma~\ref{u0u1} again). Thus, we see that $f_{\min}(w)$ contains $0^{n-1}$ in any case, and the result follows from Theorem~\ref{greedyok}.
\end{proof}

Thus, we obtain another way of generating a de Bruijn word for the set of primitive words in $\set{0,1}^n$ using a greedy algorithm. Furthermore, we see that the use of preference functions and Theorem~\ref{greedyok} give us a template to streamline the analysis of the feasibility of using greedy algorithms to generate de Bruijn words for arbitrary dictionaries. 

\subsection{Concatenation of Lyndon words}

Recall that a de Bruijn word for $\Sigma_k^n$ can also be obtained from concatenating all Lyndon words of length dividing $n$ in increasing lexicographic order. Next, we show that a de Bruijn word for the primitive words can be produced by a similar concatenation.

\begin{thm}\label{Lyndon}
Let $w$ be the concatenation of all Lyndon words in $\Sigma_{k}^n$ in increasing lexicographic order. Then $w$ is a de Bruijn word for the set of primitive words in $\Sigma_{k}^n$.
\end{thm}

Theorem~\ref{Lyndon} was first conjectured by Michael Domaratzki, who has a proof for the case $k=2$ (personal communication, July 2013). Also, throughout this section, we will let $k'$ denote the symbol $k-1$ to reduce cluttering.
 
Before we prove Theorem~\ref{Lyndon}, we need the following result due to Cummings, who previously published a proof for the case $k=2$ in~\cite{Cummings88a}. It is also implied by Duval's~\cite{Duval88a} algorithm of generating Lyndon words.

\begin{lem}\label{Lyndonlem}
Let $x \in \Sigma_{k}^n$ be a Lyndon word. Define $\l := \max \set{ i : x[i] \neq k'}$. If $\l \geq 2$, then $y := x[1\pp \l-1] (k')^{n-\l+1}$ is also a Lyndon word.
\end{lem}

\ignore{
\begin{proof}
First, we show that $y$ is primitive. Suppose it is not, and $y = z^p$ for some integer $p \geq 2$. Then $x = z^q z' z^{p-q-1}$ for some $q \in \set{0, \ldots, p-1}$, where $z'$ is $z$ with some instance of $k'$ being replaced by $x[\l]$. Since $z' < z$ and $x$ is Lyndon, $q$ has to be 0. However, this implies that $\l \leq |z|$, which means that $z= k^{n/p}$ (by equating the last $n/p$ symbols of $z^p$ and $y$), and thus $y = (k')^n$. Since $\l \geq 2$, this would imply that $x$ is not Lyndon. Thus, $y$ is primitive.

Now suppose, to obtain a contradiction, that $y$ is not Lyndon, and there exists $j \in \set{1,\ldots, n-1}$ such that $y[j+1\pp n]y[1\pp j] < y$. This would imply that
\[
x < x[j+1\pp n]x[1\pp j] < y[j+1\pp n]y[1\pp j] < y.
\]
Note that the first two inequalities are due to $x$ being Lyndon and $x < y$, respectively. Since the first $\l-1$ symbols of $x$ and $y$ coincide, it follows from the chain of inequalities that the same can be said of $x[j+1\pp n]x[1\pp j]$ and $y[j+1\pp n]y[1\pp j]$. If $2 \leq j < \l$, that would imply that $x[\l] = y[\l]$, a contradiction. However, $j \geq \l$ implies that $y[j+1\pp n]y[1\pp j]$ starts with $k'$, and $x$ being Lyndon and $\l \geq 2$ implies that $y[1] < k'$, contradicting $y[j+1\pp n]y[1\pp j] < y$.
\end{proof}
}
That is, if we replace the last non-$k'$ letter in a Lyndon word by $k'$, the resulting word is also Lyndon (unless it is $(k')^n$). We are now ready to prove Theorem~\ref{Lyndon}.

\begin{proof}[Proof of Theorem~\ref{Lyndon}]
If $w$ is the concatenation of all Lyndon words of length $n$, then $w$ has length $n$ times the number of Lyndon words in $\Sigma_k^n$. Thus, the number of circular factors of $w$ of length $n$ is equal to the number of primitive words in $\Sigma_{k}^n$, and it suffices to show that each primitive word appears at least once in $w$ (as that would imply that each primitive word appears exactly once). We do so by showing that given any Lyndon word $x$, its conjugate $x^{(i)} = x[i+1\pp n]x[1\pp i]$ appears in $w$ as a circular factor, for all $i \in \set{1,\ldots,n}$.

First, obviously $x^{(n)} = x$ appears in $w$. Next, we write $x$ as $x[1\pp \l] (k')^{n-\l}$ such that $x[\l] \neq k'$. If $\l \geq 2$, then  $y := x[1\pp\l-1] (k')^{n-\l+1}$ is also Lyndon by Lemma~\ref{Lyndonlem}. Thus, the Lyndon word that immediately follows $x$ in $w$ is sandwiched between $x$ and $y$, and has prefix $x[1\pp\l-1]$. Therefore, $w$ contains the factor $x \cdot x[1\pp\l-1]$, which contains the conjugates $x^{(1)}, x^{(2)}, \ldots, x^{(\l-1)}$.

Next, we locate the factor $x^{(i)}$ in $w$, for all $i \in \set{\l, \ldots, n-1}$. Note that $x^{(i)} = (k')^{i-\l+1} x[1\pp \l] (k')^{n-i-1}$. Let $y$ be the smallest Lyndon word that has prefix $x[1\pp\l] (k')^{n-i-1}$ (one must exist --- $x$ is one), and $z$ be the Lyndon word that immediately precedes $y$ in $w$. By the choice of $y$, $z[1\pp n+\l-i-1] < y[1\pp n+\l-i-1]$. Then by Lemma~\ref{Lyndonlem}, the last $i-\l+1$ symbols of $z$ must all be $k'$, and $zy$ contains the factor $(k')^{i-\l+1} x[1\pp \l](k')^{n-i-1} = x^{(i)}$.

The remaining case when there is no Lyndon word preceding $y$ in $w$ implies $x[1\pp \l] (k')^{n-i-1}$ is the word of all $0$s, and so $i = n-1$, and $x^{(i)} = (k')^{n-\l} 0^{\l}$. Since the first and last Lyndon words in $w$ are $0^{n-1} 1$ and $(k'-1) k'^{n-1}$ respectively, $w$ contains the circular factor $(k')^{n-1}0^{n-1}$, which must contain $x^{(i)}$. Hence, we are finished.
\end{proof}

As with the case of generating a de Bruijn word for $\Sigma_k^n$, concatenating Lyndon words is much more computationally efficient in generating a de Bruijn word for primitive words than using greedy algorithms, whose execution require exponential storage space.

\subsection{Relating de Bruijn Graphs and Lyndon Graphs}

Next, we detail yet another argument that shows the existence of de Bruijn words for primitive words. Unlike the two algorithmic proofs provided above, this argument is non-constructive, and makes use of connections between de Bruijn graphs and Lyndon graphs. 

Given integers $n,k \geq 2$, we let $\P_{n,k}$ denote the de Bruijn graph of the set of primitive words in $\Sigma_k^n$. Also, let $\L_{n,k}$ denote the \emph{Lyndon graph} of $\Sigma_k^n$, which has a vertex for each Lyndon word in $\Sigma_k^n$, and joins two Lyndon words by an edge if they differ in exactly one position. For example, Figure~\ref{fig2} illustrates the graph $\L_{6,2}$.

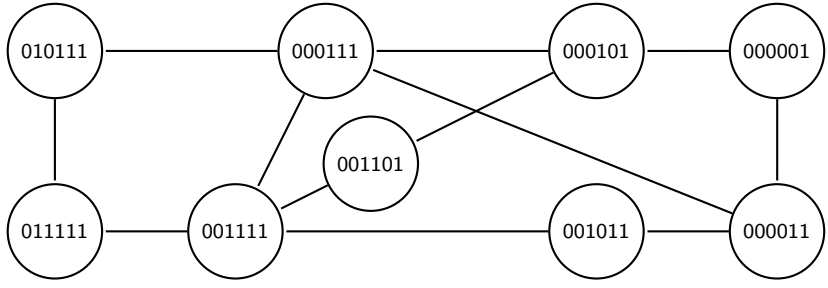
\begin{figure}
\begin{center}
\begin{tikzpicture}[scale =1.2 ,>=stealth',shorten >=1pt,auto,node distance=1cm,
  thick,main node/.style={circle, draw,font=\scriptsize\sffamily}]

  \node[main node] at (0,2) (010111) {010111};
  \node[main node] at (0,0) (011111) {011111};
  \node[main node] at (3,2) (000111) {000111};
  \node[main node] at (2,0) (001111) {001111};
  \node[main node] at (3.5,0.75) (001101) {001101};
  \node[main node] at (6,2) (000101) {000101};
  \node[main node] at (6,0) (001011) {001011};
  \node[main node] at (8,2) (000001) {000001};
  \node[main node] at (8,0) (000011) {000011};

 \path 
(000001) edge (000011)
(000001) edge (000101)
(000011) edge (000111)
(000011) edge (001011)
(000101) edge (000111)
(000101) edge (001101)
(000111) edge (001111)
(000111) edge (010111)
(001011) edge (001111)
(001101) edge (001111)
(001111) edge (011111)
(010111) edge (011111);
\end{tikzpicture}
\caption{The Lyndon graph $\L_{6,2}$}\label{fig2}
\end{center}
\end{figure}
Notice that $\L_{6,2}$ only has one component. In fact, this is shown by Cummings to be true in general~\cite{Cummings88a}.

\begin{lem}\label{Lyndoncon}
$\L_{n,k}$ is connected for all $n,k \geq 2$.
\end{lem}

\begin{proof}
Given any pair of Lyndon words $x,y \in \Sigma_{k}^n$, Lemma~\ref{Lyndonlem} shows that there is a path from $x$ to $x[1] (k')^{n-1}$ in $\L_{n,k}$. Similarly, there is also a path between $y$ and $y[1](k')^{n-1}$. Since $x[1](k')^{n-1}$ is adjacent to $y[1](k')^{n-1}$, we see that there is a path between $x$ and $y$ in $\L_{n,k}$. Thus, $\L_{n,k}$ is connected.
\end{proof}

On the surface, $\P_{n,k}$ and $\L_{n,k}$ appear to have very little in common. First of all, the former is directed and the latter is not. Also, their vertices are represented by words of different lengths, with adjacency rules that are quite different. However, it turns out that they can be related through a series of basic graph operations.

Given a directed graph $G$, its \emph{line graph} $L(G)$ is obtained by defining a vertex for each arc in $G$, and joining $u$ and $v$ in $L(G)$ if there is a vertex in $G$ that is incident with their corresponding arcs. Note that while $G$ is directed, $L(G)$ is undirected. Next, let $G$ be an undirected graph and $S \subseteq V(G)$. Then \emph{contracting} $S$ in $G$ yields the graph obtained from replacing the vertices in $S$ by a single vertex $v_S$, and joining it to vertices in $V(G) \setminus S$ that was adjacent to some vertex in $S$.

Then we have the following:

\begin{prop}
Let $\H_{n,k}$ be the graph obtained from starting with $L(\P_{n,k})$, and successively contracting $\set{x^{(i)} : i \in \set{1,\ldots,n}}$ for all Lyndon words $x \in \Sigma_k^n$. Then $\L_{n,k}$ is a subgraph of $\H_{n,k}$.
\end{prop}

\begin{proof}
First, if during the contraction process, we label the vertex obtained from contracting $\set{x^{(i)} : i \in \set{1,\ldots,n}}$ by $x$ for all Lyndon word $x \in \Sigma_k^n$, then it is easy to see that $\H_{n,k}$ and $\L_{n,k}$ have the same vertex set. Thus, it suffices to show that two Lyndon words are joined by an edge in $\H_{n,k}$ if they differ by exactly one position.

Let $u \a v$ and $u \b v$ be two Lyndon words in $\Sigma_k^n$, where $u,v \in \Sigma_k^*$ and $\a, \b \in \Sigma_k$. Observe that $\a vu$ and $vu \b$ are both arcs in $\P_{n,k}$ (since they are both primitive), and share the vertex $vu$. Hence, $\a vu$ and $vu \b$ are joined by an edge in $L(\P_{n,k})$. Since $\a vu$ and $vu \b$ are conjugates of $u \a v$ and $u \b v$ respectively, we see that $u\a v$ and $u\b v$ are joined by an edge in $\H_{n,k}$.
\end{proof}

Figure~\ref{fig3} illustrates the transformation from $\P_{4,2}$ to $\H_{4,2}$, which turns out to be exactly the graph $\L_{4,2}$. In general, while $\H_{n,k}$ and $\L_{n,k}$ have the same vertex set, the former can have more edges. For instance, while $000011$ and $001101$ differ by three positions, they are adjacent in $\H_{6,2}$, since the arcs $000110$ and $001101$ share the vertex $00110$ in $\P_{6,2}$. 

\begin{figure}[htb]
\begin{center}
\begin{tikzpicture}[scale =0.8 ,>=stealth',shorten >=1pt,auto,node distance=4cm,
  thick,main node/.style={circle, draw,font=\scriptsize\sffamily}, word node/.style={font=\footnotesize\sffamily}]

  \node[main node] at (2,0.5) (000) {000};
  \node[main node] at (0,4) (001) {001};
  \node[main node] at (2,2) (010) {010};
  \node[main node] at (0,6) (011) {011};
  \node[main node] at (4,4) (100) {100};
  \node[main node] at (2,8) (101) {101};
  \node[main node] at (4,6) (110) {110};
  \node[main node] at (2,9.5) (111) {111};
;
 \path[->, font=\scriptsize\sffamily] 
(000) edge node [below left] {$0001$} (001)
(001) edge node [right] {$0011$} (011)
(001) edge node [above right=-0.1cm] {$0010$} (010)
(010) edge node [above left=-0.1cm] {$0100$} (100)
(100) edge node [below right] {$1000$} (000)
(011) edge node [below] {$0110$} (110)
 (011) edge node [above left] {$0111$} (111)
 (100) edge node [above	] {$1001$} (001)
 (101) edge node [below right=-0.1cm]{$1011$} (011)
 (110) edge node [left]{$1100$} (100)
 (110) edge node [below left=-0.1cm]{$1101$} (101)
(111) edge node [above right] {$1110$} (110);

\def\x{8.8}
\def\y{-1}
\def\z{15.5}

  \node[main node] at (\x + 0,1) (0001) {0001};
  \node[main node] at (\x + 4,1) (1000) {1000};
  \node[main node] at (\x + 0.7,2.5) (0010) {0010};
  \node[main node] at (\x + 3.3,2.5) (0100) {0100};
  \node[main node] at (\x + 2,4) (1001) {1001};
  \node[main node] at (\x + 0,5) (0011) {0011};
  \node[main node] at (\x + 4,5) (1100) {1100};
  \node[main node] at (\x + 2,6) (0110) {0110};
  \node[main node] at (\x + 0.7,7.5) (1011) {1011};
  \node[main node] at (\x + 3.3,7.5) (1101) {1101};
  \node[main node] at (\x + 0,9) (0111) {0111};
  \node[main node] at (\x + 4,9) (1110) {1110};

\draw[dotted] (\x - 3.2,3.1) -- (\x + 4.8,3.1) -- (\x + 4.8, 0.4) -- (\x -3.2, 0.4) -- (\x -3.2,3.1) ;
\node[word node] at (\x -1.9,2.2) (1) {\tn{Conjugates}};
\node[word node] at (\x -1.9,1.8) (1) {\tn{of 0001}};
\draw[dotted] (\x - 3.2,6.6) -- (\x + 4.8,6.6) -- (\x + 4.8, 3.4) -- (\x -3.2, 3.4) -- (\x -3.2,6.6) ;
\node[word node] at (\x -1.9,5.2) (1) {\tn{Conjugates}};
\node[word node] at (\x -1.9,4.8) (1) {\tn{of 0011}};
\draw[dotted] (\x - 3.2,9.6) -- (\x + 4.8,9.6) -- (\x + 4.8, 6.9) -- (\x -3.2, 6.9) -- (\x -3.2,9.6) ;
\node[word node] at (\x -1.9,8.2) (1) {\tn{Conjugates}};
\node[word node] at (\x -1.9,7.8) (1) {\tn{of 0111}};

\node at (2,\y) (P) {$\P_{4,2}$};
\node at (\x+2,\y) (LP) {$L(\P_{4,2})$};
\node at (\z,\y) (H) {$\H_{4,2}$};

 \path[->, dashed] 
(P) edge node [below] {$L(\cdot)$} (LP)
(LP) edge node [below] {\tn{Contraction}} (H);


  \node[main node] at (\z,1.5) (0001b) {0001};
  \node[main node] at (\z,5) (0011b) {0011};
  \node[main node] at (\z,8.5) (0111b) {0111};

 \path[font=\scriptsize\sffamily] 
(0001b) edge (0011b)
(0011b) edge (0111b);

 \path[font=\scriptsize\sffamily] 
(0001) edge (0010)
(0010) edge (0100)
(0100) edge (1000)
(1000) edge (0001)
(1001) edge (0100)
(1001) edge (0010)
(1001) edge (0011)
(1100) edge (1000)
(0011) edge (0001)
(0011) edge (0110)
(0110) edge (1100)
(1100) edge (1001)
(0110) edge (1101)
(0110) edge (1011)
(1100) edge (1110)
(0011) edge (0111)
(1011) edge (0111)
(0111) edge (1110)
(1110) edge (1101)
(1101) edge (1011)
(0110) edge[bend right] (0111)
(0110) edge[bend left] (1110)
(1100) edge (1101)
(0011) edge (1011)
(1100) edge (0100)
(0011) edge (0010)
(1001) edge[bend right] (1000)
(1001) edge[bend left] (0001)

;

\end{tikzpicture}
\caption{Transforming $\P_{4,2}$ to $\H_{4,2}$}\label{fig3}
\end{center}
\end{figure}
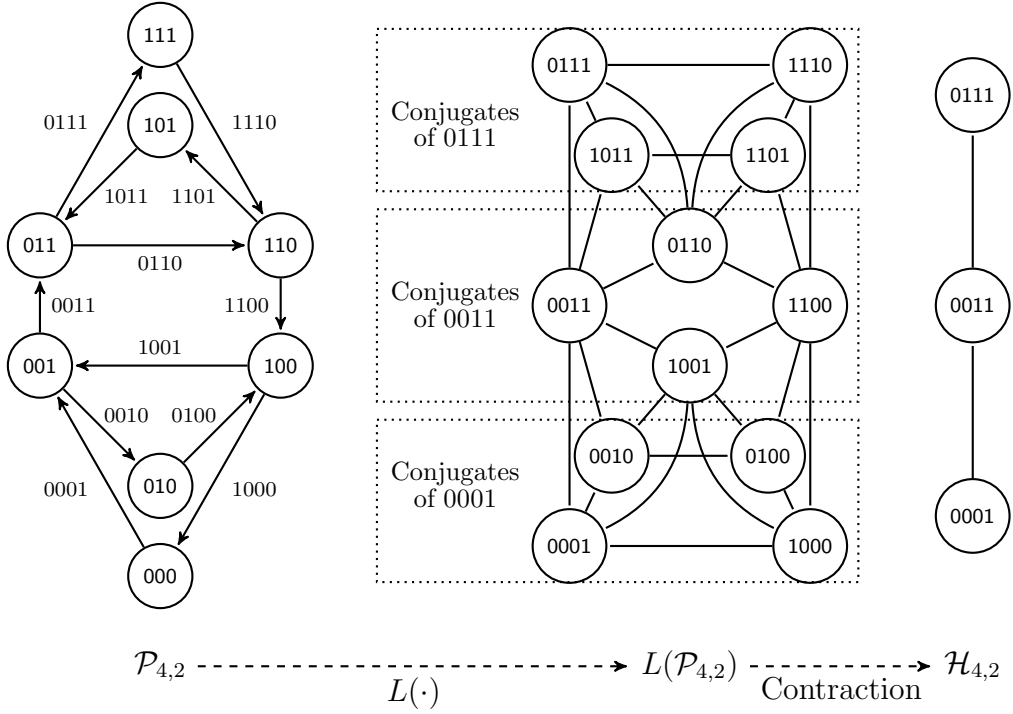
Now we assemble the results in this section to provide yet another proof that a de Bruijn word for the primitive words exists, and we do that by showing that $\P_{n,k}$ has an Eulerian cycle.

First, Lemma~\ref{C2prim} implies that every vertex in $\P_{n,k}$ has the same in-degree and out-degree. Thus, it suffices to show that the underlying undirected graph of $\P_{n,k}$ is connected. 

To obtain a contradiction, suppose there are vertices $u,v$ that belong to different components in $\P_{n,k}$. If we let $x$ and $y$ be arcs that are incident with $u,v$ respectively, then $x$ and $y$ are in different components in $L(\P_{n,k})$. Next, observe that the $n$ conjugates of any primitive word form a directed cycle of length $n$ in $\P_{n,k}$. Thus, the $n$ corresponding vertices cannot be spread across multiple components in $L(\P_{n,k})$, and hence $\H_{n,k}$ cannot have fewer components than $L(\P_{n,k})$. 

However, $\L_{n,k}$ is shown to be connected, is contained in $\H_{n,k}$, and they have the same vertex set. Therefore, $\H_{n,k}$ only has one component, which implies that $L(\P_{n,k})$ is connected, a contradiction. Hence, we conclude that $\P_{n,k}$ has an Eulerian cycle, and there is a de Bruijn word for the set of primitive words in $\Sigma_k^n$.

In fact, if we extract the minimal ingredients we used the above argument, we obtain the following slightly stronger statement:

\begin{cor}
Let $\D \subseteq \Sigma_k^n$ be a dictionary that satisfies~\eqref{C2}, and has the property that for every pair of Lyndon words $u \a v, u \b v \in \Sigma_k^n$ where $u,v \in \Sigma_k^*$ and $\a,\b \in \Sigma_k$, 
\ignore{
\[
\D \cap \set{ \left(u\a v\right)^{(i)} : i \in \set{1,\ldots,n}} \neq \emptyset \quad \tn{and}  \quad \D \cap \set{ \left(u\b v\right)^{(i)} : i \in \set{1,\ldots,n}} \neq \emptyset,
\]
then}
\[
\D \cap \set{\a vu, vu\a} \neq \emptyset \quad \tn{and} \quad \D \cap \set{\b vu, vu\b} \neq \emptyset.
\]
Then there is a de Bruijn word for $\D$.
\end{cor}

\begin{proof}
Consider the de Bruijn graph $G^{\D}$, and let $H$ be the graph obtained from contracting all the conjugate classes of the line graph of $G^{\D}$. Notice that the Lyndon words $u \a v, u\b v$ differ by exactly one bit, and thus are adjacent in $\L_{n,k}$. Now if $\D \cap \set{\a vu, vu\a} \neq \emptyset$ and $\D \cap \set{\b vu, vu\b} \neq \emptyset$, that means $\D$ contains a conjugate of $u\a v$ and a conjugate of $u\b v$ such that those two edges are both incident with the vertex $vu$ in $G^{\D}$. As a result, $u\a v$ and $u\b v$ are joined by an edge in $H$, and thus $H$ contains $\L_{n,k}$ as a subgraph. This implies that $H$ is connected, and consequently the underlying undirected graph of $G^{\D}$ is connected. Together with the fact that $\D$ satisfies~\eqref{C2}, we conclude that $\D$ has a de Bruijn word.
\end{proof}

It would be interesting to know if any other properties of primitive words (or other families of words) and Lyndon words can be uncovered by this relation between their corresponding graphs. Establishing a tighter connection between these families of graphs (e.g. finding a transformation on $\P_{n,k}$ that yields exactly $\L_{n,k}$) could also lead to new and interesting findings.

\section{Short sequences containing powers}

While an arbitrary dictionary $\D$ may not have a de Bruijn word, there might be words of length not much larger than $|\D|$ that contains all words in $\D$ as circular factors. For instance, while we mentioned in the previous section that $\D := \set{0000,0001,0011,0111}$ does not have a de Bruijn word, there are many sequences that contain all fours words in $\D$ as factors, with $0000111$ being the shortest such sequence. Thus, in this regard, we can consider the word $0000111$ as the closest thing to a de Bruijn word for $\D$, as there are no shorter sequences that contain all words in $\D$.

This motivates the following question: Given an arbitrary dictionary $\D \subseteq \Sigma_k^n$, what is the shortest word that contains all words in $\D$ as circular factors? Such a sequence can be seen as a generalization of de Bruijn words, since if a dictionary $\D$ has a de Bruijn word, that word must also be the shortest possible sequence that contains all words in $\D$ as circular factors.

In this section, we tackle the above question for a particular family of dictionaries, and try to find the shortest sequence that contains all $p$-powers in $\Sigma_k^{pn}$ as circular factors. For $p=1$, it is obvious that there is a de Bruijn word for all $p$-powers (it would just be a de Bruijn word for $\Sigma_k^n$). However, this does not apply for any $p>1$. For instance, $\D = \set{0000,0101, 1010, 1111}$ are the set of all squares in $\set{0,1}^4$, and the shortest sequence that contains all four words as circular factors is $w=000010101111$, which has length $12$. More generally, if we let $\mathcal{D}$ to be the set of $p$-powers in $\Sigma_k^{pn}$, then $G^{\D}$ has as many components as the number of conjugacy classes in $\Sigma_k^n$. In fact, we shall soon see that any sequence that contains all $k^n$ $p$-powers in $\Sigma_k^{pn}$ must contain at least $(p-1)k^n$ factors of length $pn$ that are not $p$-powers.

Define an equivalence relation on $\Sigma_k^n$, where $u \sim v$ if and only if they are conjugates of each other, and let $C(n,k)$ denote the number of conjugacy classes in $\Sigma_k^n$. It is well known that $C(n,k) =  \sum_{ d \geq 1: d | n} \frac{\phi(d)}{n} k^{\frac{n}{d}}$, where $\phi(d)$ is Euler's totient function --- the number of integers between $1$ and $d$ that are coprime with $d$. Note that $C(n,k) \geq \frac{k^n}{n}$ for all $n,k$.

Then we have the following:

\begin{prop}\label{lb}
Suppose $w \in \Sigma_k^*$ contains every $p$-power in $\Sigma_k^{pn}$ as factors. Then $|w| \geq k^n + (p-1)nC(n,k) \geq pk^n$.
\end{prop}

\begin{proof}
Given $x, y \in \Sigma_k^n$, observe that if $x \not\sim y$, then any word that contains both $x^p$ and $y^p$ as factors has length at least $2pn-n+1$. Therefore, every time two consecutive $p$-powers in $w$ belong to different conjugacy classes, there are at least $(p-1)n$ factors of length $pn$ in $w$ in between that are not $p$-powers. Since there are $C(n,k)$ conjugacy classes in $\Sigma_k^n$, we see that $w$ contains at least $(p-1)n (C(n,k) -1)$ factors of length $pn$ that are not $p$-powers.

Since $w$ must also contain at least $k^n$ factors that are $p$-powers, there are a total of at least $k^n + (p-1)n(C(n,k)-1))$ factors of length $pn$ in $w$. Hence
\[
|w| \geq k^n + (p-1)n(C(n,k) -1) + pn-1 \geq k^n + (p-1)nC(n,k) \geq p k^n,
\]
and our claim follows.
\end{proof}

Next, we show that there is a word $w$ of length $\approx (p + \frac{1}{k})k^n$ over $\Sigma_k$ that contains all $p$-powers of length $pn$. Given $u \in \Sigma_k^n$, define 
\[
\delta(u) := \frac{\min \set{ i \geq 1 : u^{(i)} = u} }{n}.
\]
Equivalently, $\delta(u)$ is the reciprocal of $\max \set{p \geq 1: \tn{$u$ is a $p$-power}}$. Note that $\delta(u) = 1$ if and only if $u$ is primitive, and that $u^{p+\delta(u) - (1/n)}$ contains all $p$-powers of all conjugates of $u$ as factors exactly once.

Next, we say that a word $s \in \Sigma_k^{*}$ is a \emph{conjugate cover} of $\Sigma_k^n$ if for every $u \in \Sigma_k^n$, $s$ contains some circular factor of length $n-1$ in $u$. Conjugate covers exist for all $n,k$. For instance, if we take $t$ to be a de Bruijn word for $\Sigma_k^{n-1}$, then $s := t \cdot t[1\pp n-2]$ is a conjugate cover, since it contains all words in $\Sigma_k^{n-1}$ as factors. We then construct a word $w$ that contains all $p$-powers in $\Sigma_k^{pn}$ by the following algorithm:

\begin{alg}\label{algb}
Generating a sequence $w$ that contains all $p$-powers in $\Sigma_k^{pn}$

\begin{algorithm}[H]
\DontPrintSemicolon 
\KwIn{Integers $n,k,p$ where $n,k \geq 2, p \geq 1$, and $s$ a conjugate cover of $\Sigma_k^{n}$}
Set $w = \epsilon$ (the empty string)\;
Set $L = \Sigma_k^n$\;
\For{$j = 1, \ldots, |s|-n+2$}{
\For{$\a = 0,1,\ldots,k-1$}{
Set $u = s[j\pp j+n-2]\a$\;
\eIf{$u \in L$}{\emph{Accept} $\a$ and append $u^{p + \delta(u) -1}$ to the end of $w$ \;
Remove all conjugates of $u$ from $L$\;
}{\emph{Reject} $\a$ and do not append anything}
}
Append $s[j]$ to $w$\;
}
Append $s[|s| -n+3\pp |s|]$ to $w$\;
\Return{$w$}

\end{algorithm}
\end{alg}

For example, consider the case $n=k=3$ and $p=2$. The word $s := 0221201100$ is a conjugate cover of $\set{0,1,2}^3$. In this case, Algorithm~\ref{algb} would execute as follows:

\begin{center}
\begin{tabular}{|c|c|l|l|l|}
\hline
$j$ & $s[j \pp j+1]$ & Accepted $\a$'s & Append to $w$ & Removed from $L$  \\
\hline
$1$ & $02$ & $0,1,2$ & $0200200210210220220$ & Conjugates of $020,021,022$ \\
$2$ & $22$ & $1,2$ & $22122122222$ & Conjugates of $221,222$ \\
$3$ & $21$ & $1$ & $2112112$ & Conjugates of $211$\\
$4$ & $12$ & $0$ & $1201201$ & Conjugates of $120$\\
$5$ & $20$ & None & $2$ & None\\
$6$ & $01$ & $0,1$& $0100100110110$ & Conjugates of $010,011$\\
$7$ & $11$ & $1$ & $11111$ & $111$ \\
$8$ & $10$ & None & $1$ & None\\
$9$ & $00$ & $0$ & $00000$ & $000$\\
\hline
\end{tabular}
\end{center}

The algorithm finally appends $0$ (the last symbol of $s$) to $w$, and outputs the word
\begin{eqnarray*}
w &=& 0200200210210220220~ 22122122222~2112112~1201201~2~\\
&& 0100100110110~11111~1~00000~0,
\end{eqnarray*}
which contains all squares of length $6$ over $\set{0,1,2}$. Next, we show that the word generated by Algorithm~\ref{algb} is not ``too much'' longer than the lower bound shown in Proposition~\ref{lb}.

\begin{thm}\label{alg2}
Let $w$ be the word constructed by Algorithm~\ref{algb}. Then $w$ contains $x^p$ as a factor for all $x \in \Sigma_k^n$. Moreover, $|w| = k^n + (p-1)nC(n,k) + |s|$.
\end{thm}

\begin{proof}
Recall that, given $x \in \Sigma_k^n$, $x^{(i)} = x[i+1\pp n]x[1\pp i]$. We first prove that each $p$-power appears in $w$ at least once by showing that for every $x \in \Sigma_k^n$, there exists $i \in \set{1,\ldots,n}$ such that $w$ contains $(x^{(i)})^{p+ \delta(x) - (1/n)}$ as a factor.

Let $j$ be the smallest index such that $s[j\pp j+n-2]$ is a prefix of some conjugate of $x$, say $x^{(i)}$. Since $s$ is a conjugate cover, such an index $j$ must exist. Then we know that the algorithm would accept $\a = x^{(i)}[n]$ at step $j$, and $(x^{(i)})^{p-1 + \delta(x)}$ is appended to $w$.

If at step $j$, some symbol larger than $\a$ is accepted, then we know the block $s[j\pp j+n-2] = x^{(i)}[1\pp n-1]$ immediately follows, giving us the desired power of $x^{(i)}$. Otherwise, we know that $s[j]$ gets added to $w$ at the end of step $j$.

Then, if any symbol is accepted in step $j+1$, then $s[j+1\pp j+n-1]$ is added to $w$, and we get our desired power of $x^{(i)}$. Otherwise, we just add $s[j+1]$ at the end of step $j+1$. Proceeding in this manner, we see that the algorithm always adds $s[j\pp j+n-2]$ immediately after adding $(x^{(i)})^{p-1 + \delta(x)}$ at step $j$. Since this holds for all $x \in \Sigma_k^n$, we see that $w$ contains all $p$-powers in $\Sigma_k^{pn}$. 

Next, we compute $|w|$. We have already found $k^n$ factors of length $pn$ that are $p$-powers. To count the other factors in $w$, we need to observe that, after accepting $\a_1$ at step $j$, if the next symbol accepted by the algorithm is $\a_2$ during step $j+\l$, then there are exactly $(p-1)n+\l$ factors of length $pn$ in $w$ between the last $p$-power in $(s[j\pp j+n-2]\a_1)^{p + \delta}$ and the first $p$-power in $(s[j+\l\pp j+n+\l-2]\a_2)^{p+\delta}$. Note that there could be $p$-powers among these blocks (e.g. when $\l = 1$ and $\a_2 = s[j+n-1]$), but we nonetheless count them under the ``other factors'' category. Also, if the last symbol accepted by Algorithm~\ref{algb} is $\a$ at step $|s| - n+2 - \l$, then there are $\l$ factors of length $pn$ in $w$ after the last $p$-power in $u^{p+\delta(u)}$, where $u = s[|s|-n+2-\l\pp|s|-\l]\a$.

Since each symbol accepted by Algorithm~\ref{algb} corresponds to a unique conjugacy class in $\Sigma_k^n$, we see that a total of $C(n,k)$ symbols are accepted throughout the algorithm. Therefore, $w$ contains exactly  $(p-1)n (C(n,k)-1) + |s|-n+1$ of these ``other factors'' of length $pn$. Thus,
\begin{eqnarray*}
|w| &=& k^n + (p-1)n(C(n,k)-1) + (|s|-n+1) + pn-1\\
& =& k^n + (p-1)nC(n,k)  + |s|,
\end{eqnarray*}
and we are finished.
\end{proof}

As mentioned before, we can always construct a conjugate cover out of a de Bruijn word for $\Sigma_{k}^{n-1}$. In fact, we could do slightly better than that when $n-1$ is not prime:

\begin{cor}\label{squarecor}
Suppose $n,k \geq 2$, and let $\D$ be the set of primitive words in $\Sigma_k^{n-1}$. Then there exists a word $w$ of length $k^n + (p-1)nC(n,k) + |\D|  +n+k-2$ that contains all $p$-powers in $\Sigma_k^{pn}$ as factors.
\end{cor}

\begin{proof}
By Theorem~\ref{alg2}, it suffices to show that there is a conjugate cover of $\Sigma_{k}^{n}$ of length $|\D| + n+k-2$. Let $t$ be the de Bruijn word for $\D$ constructed by concatenating Lyndon words as described in Theorem~\ref{Lyndon}. Then $|t| = |\D|$, and $t$ contains $\a^{n-2}$ as a factor at least once for all $\a \in \Sigma_k$. We obtain $s$ by replacing an instance of $\a^{n-2}$ in $t$ by $\a^{n-1}$ for each $\a \in \Sigma_k$, and then appending $0^{n-2}$ at the end. It is easy to see that $|s| = |\D| + n+k-2$, and $s$ contains all words in $\D$, as well as $\a^{n-1}$ for all $\a \in \Sigma_k$, as factors.

To show that $s$ is a conjugate cover, it suffices to show that for all $u \in \Sigma_k^n$, either it has a circular factor of length $n-1$ that is primitive, or $u = \a^n$ for some symbol $\a$. Observe that, for any $i \in \set{1,\ldots,n}$, if neither $u[i+1\pp n]u[1\pp i-1]$ nor $u[i+2\pp n]u[1\pp i]$ is primitive, then $u[i+1] = u[i]$ by Lemma~\ref{C2prim} and~\ref{u0u1}. Applying this argument on all $i$ yields that $u = \a^n$ for some $\a \in \Sigma_k$, and it follows that $s$ is a conjugate cover.
\end{proof}

Since the number of primitive words in $\Sigma_k^{n-1}$ is less than $k^{n-1}$, we have now shown that the shortest sequence that contains all $p$-powers in $\Sigma_k^{pn}$ has length roughly between $pk^n$ and $(p+ \frac{1}{k}) k^n$. For $p=1$, we know the truth is much closer to the lower bound, as there is a word of length $k^n +n-1$ that contains all words in $\Sigma_k^n$ as factors --- any de Bruijn word of $\Sigma_{k}^n$ with the first $n-1$ symbols repeated at the end would do.

Computational evidence suggests that this seems to be the case for $p=2$ as well.  Suppose we consider the special case of $k=p=2$, and build a sequence that contains all squares in $\set{0,1}^{2n}$ by the following procedure: 

\begin{alg}\label{algc}
Constructing a word $w$ that contains all squares of length $2n$ over $\set{0,1}$
\begin{algorithm}[H]
\DontPrintSemicolon 
\KwIn{Integer $n \geq 2$}
Set $w = 0^{2n}$\;
Set $L = \set{0,1}^n$\;
\While{$L \neq \emptyset$}
{
Pick $u \in L$ such that the prefix of $u$ overlaps the most with the current suffix of $w$. If there is a tie, pick the lexicographically smallest $u$. Append to $w$ such that $w$ now has suffix $u^{2 + \delta(u) - (1/n)}$. \;
Remove all conjugates of $u$ from $L$\;
}
\Return{$w$}
\end{algorithm}
\end{alg}

For any integer $n$, let $g(n)$ be the length of the sequence obtained by Algorithm~\ref{algc}, and let $f(n) := \frac{g(n)}{2^n + n C(n,2)}$. Figure~\ref{fig4} illustrates the behaviour of $f(n)$ for $n \in \set{4,\ldots,25}$.

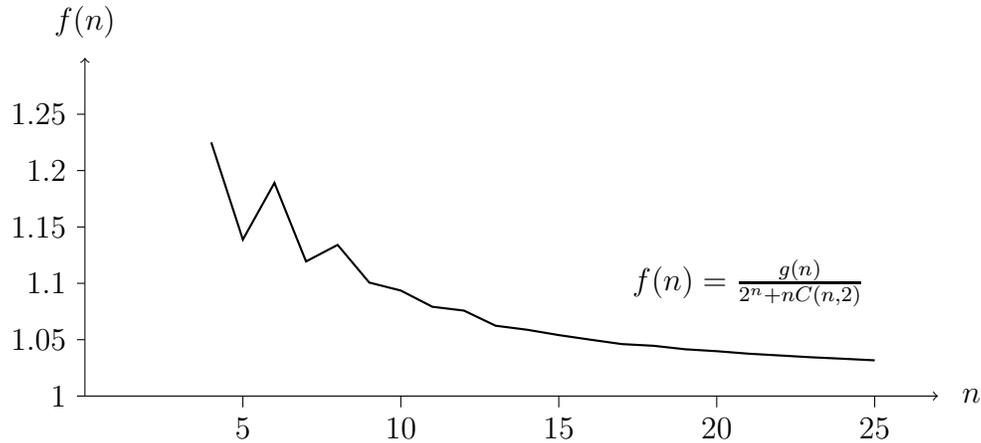
\begin{figure}
\begin{center}

\begin{tikzpicture}[y=15cm, x=.42cm]
	\draw[->] (0,1) -- coordinate (x axis mid) (27,1);
    	\draw[->] (0,1) -- coordinate (y axis mid) (0,1.3);
    	\foreach \x in {5,10,...,25}
     		\draw (\x,1.00) -- (\x,0.99)
			node[anchor=north] {\x};
    	\foreach \y in {1,1.05,1.1,1.15,1.2,1.25}
     		\draw (0,\y) -- (-3pt,\y) 
     			node[anchor=east] {\y}; 
	\node[right = 0.18cm] at (27,1) {$n$};
	\node[above=0.12cm] at (0,1.3) {$f(n)$};


	\draw[thick] (4,1.225) -- (5,1.1388) -- (6,1.1891) -- (7,1.1194) -- (8,1.1341) -- (9,1.1007) -- (10,1.0936) -- (11,1.0792) -- (12,1.0758) -- (13,1.0624) -- (14,1.0588) -- (15,1.0541) -- (16,1.0500) -- (17,1.0461) -- (18,1.0446) -- (19,1.0415) -- (20,1.0398) -- (21,1.0376) -- (22,1.0360) -- (23,1.0344) -- (24,1.0331) -- (25,1.0317);    \node at (21,1.1) {$f(n) = \frac{g(n)}{ 2^n + nC(n,2)}$};
\end{tikzpicture}
\caption{Computational results for $f(n)$}\label{fig4}
\end{center}
\end{figure}

By Corollary~\ref{squarecor}, the length of shortest word that contains all squares in $\set{0,1}^{2n}$ is bounded above by roughly $\frac{5}{4}\left(2^n + nC(n,2)\right)$. However, we see that $f(n)$ appears to approach 1 as $n$ increases, and there seems to be room for improvement for the upper bound. Perhaps constructing the shortest possible conjugate covers can improve the upper bound to, say,  $k^n + (p-1)nC(n,k) + O(k^{n-1} /n)$. Also, we remark that the lower bound in Proposition~\ref{lb} also holds for fractional powers $p$ (given a positive real number $p$ where $pn$ is an integer, we can define $x^p := x^{\lfloor p \rfloor} x[1 \pp (p - \lfloor p \rfloor)n]$). It would be interesting to know if ``short'' sequences that contains all $p$-powers for a fractional $p$ exist, and whether there are efficient algorithms that generate short sequences that contains all $p$-powers in general.

\section{Acknowledgements}

We would like to deeply thank Jeffrey Shallit, who brought to our attention the problems tackled in this manuscript. In particular, it was his suggestion that a greedy algorithm could be applied to generate a de Bruijn word for primitive words. He also provided many helpful comments on the earlier drafts of this manuscript.

Furthermore, we would like to express our gratitude towards the anonymous referees who reviewed this manuscript, and gave extremely detailed and helpful suggestions that have improved both the content and the presentation of this paper.

Finally, some of the findings in this manuscript were obtained while the author was at the University of Waterloo, supported in part by an NSERC Scholarship, a Tutte Scholarship and a Sinclair Scholarship.

\bibliographystyle{alpha}
\bibliography{ref}

\end{document}